\documentclass{amsart}
\usepackage{graphicx}
\usepackage{amsmath}
\usepackage{amssymb}
\usepackage{amsthm}
\usepackage{a4wide}

\theoremstyle{plain}
\newtheorem{thm}{Theorem}

\theoremstyle{definition}

\newtheorem{quest}{Question}

\newcommand{\R}{\ensuremath{\mathbb{R}}}

\newcommand{\T}{\ensuremath{{\mathcal T}}}

\renewcommand{\epsilon}{\varepsilon}

\begin{document}

\title[Noncongruent equidissections]{Noncongruent equidissections of the plane}
\author{D.~Frettl{\"o}h}
\address{Bielefeld University, Postfach 100131, 33501 Bielefeld, Germany}
\date{\today}

\begin{abstract}
Nandakumar asked whether there is a tiling of the plane by pairwise 
non-congruent triangles of equal area and equal perimeter. Here a
weaker result is obtained: there is a tiling of the plane by pairwise
non-congruent triangles of equal area such that their perimeter is
bounded by some common constant. Several variants of the problem
are stated, some of them are answered.
\end{abstract}

\maketitle

\begin{center}
{\em Dedicated to Karoly Bezdek and to my academic brother Egon Schulte
on the occasion of their 60th birthday}
\end{center}

\section{Introduction}

There are several problems in Discrete Geometry, old and new, that can be 
stated easily but are hard to solve. Tilings and dissections provide a large 
number of such problems, see for instance \cite[Chapter C]{CFG}. On his blog
\cite{Nblog}, R.~Nandakumar asked in 2014:
\begin{quest} \label{q:frage}
``Can the plane be filled by triangles of same 
area and perimeter with no two triangles congruent to each other?''
\end{quest}
His webpage \cite{Nblog} contains several further interesting problems 
of this flavour. The main result of this paper, Theorem \ref{thm:tritil}, 
answers a weaker form of the question above affirmatively. Section 
\ref{sec:tritil} is dedicated to the statement and the proof of this result,
together with its analogues for quadrangles, pentagons and hexagons. 
Section \ref{sec:more} formulates several variants of this problem and 
gives a systematic overview. Section \ref{sec:basic} contains some basic
observations and a first result on a similar result for quadrangles. 

\subsection{Notation}
A {\em tiling} of $\R^2$ is a collection $\{T_1, T_2, \ldots \}$ 
of compact sets $T_i$ (the {\em tiles}) that is a packing of $\R^2$ 
(i.e., the interiors of distinct tiles are disjoint) as well as a 
covering of $\R^2$ (i.e. the union of the tiles equals $\R^2$).
In general, tile shapes may be pretty complicated, but for
the purpose of this paper tiles are always simple convex polygons. 
A tiling is called {\em vertex-to-vertex}, if the intersection of any two 
tiles is either an entire edge of both tiles, or a point, or empty. 
A tiling $\T$ is {\em locally finite} if any compact set in $\R^2$ 
intersects only finitely many tiles of $\T$. A tiling $\T$ is 
{\em normal} if there are $R>r>0$ such that (1) each 
tile in $\T$ contains some ball of radius $r$, and (2) each tile 
is contained in some ball of radius $R$. By \cite[3.2.1]{GS} we 
have that each normal tiling is locally finite.

\begin{figure}
\includegraphics[width=.5\textwidth]{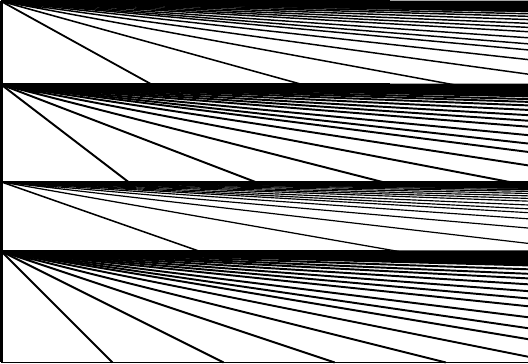}
\caption{Tiling of the plane by
pairwise non-congruent triangles of unit area. The perimeters of the 
triangles are unbounded. Moreover, the tiling is not locally finite.
\label{fig:equi-tri-strip}}
\end{figure}

\section{Basic observations} \label{sec:basic}

In Question \ref{q:frage} above, ``filled'' is to be understood in 
the sense that the plane is
covered without overlaps. In other words: is there a tiling of the plane 
by pairwise noncongruent triangles all having the same area and the same 
perimeter? If one tries to find a solution one realises that the problem
seems to be highly overdetermined. One possibility to relax the problem
is to drop the requirement on the perimeter. So one may ask ``Is there a 
tiling of the plane by pairwise noncongruent triangles all having the 
same area?'' It is not too hard to find examples. One possibility
is to partition the plane into half-strips and divide these half-strips
into triangles, as shown in Figure \ref{fig:equi-tri-strip}.

The image indicates how to fill the right half-plane by half-strips
made of triangles. If all half-strips have different widths then all 
triangles are distinct. The left half-plane can be filled in
an analogous manner. Anyway, this tiling is not locally finite:
the upper vertex of any half-strip is contained in infinitely many
triangles. So one may ask ``Is there a locally finite tiling of 
the plane by pairwise noncongruent triangles of unit area?''
Even in this stronger form the question was already answered by 
R.~Nandakumar. The image in Figure \ref{fig:equi-tri-thin} 
shows a solution, see also \cite{Wblog}. The idea is to dissect the
upper right quadrant into triangles of area $\frac{1}{2}$ by zigzagging 
between the horizontal axis and the vertical axis. The triangles become
very long and thin soon. Nevertheless they are filling the quadrant. 
For the remaining
three quadrants one uses an analogous construction, perturbing the
coordinates slightly. (For instance, stretch a copy of the first
quadrant by some irrational factor $q>1$ in the horizontal direction and
shrink it by $\frac{1}{q}$ in the vertical direction. See \cite{Wblog}
for an alternative, more detailed explanation.) 
\begin{figure}[b]
\includegraphics[width=.9\textwidth]{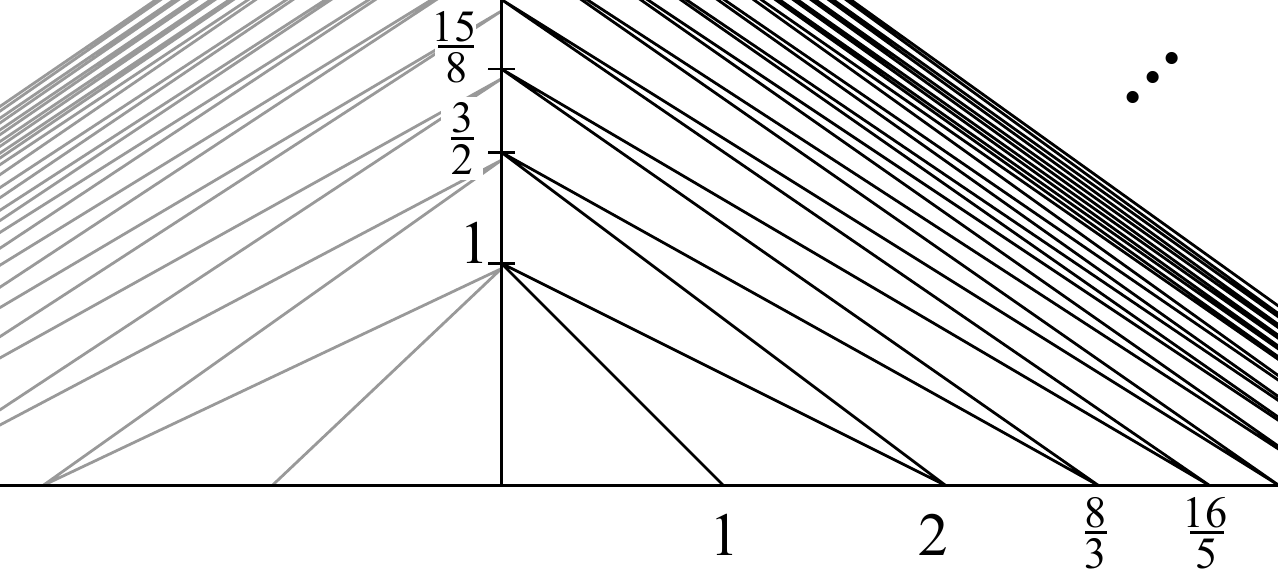}
\caption{Tiling of the plane by
pairwise non-congruent triangles of area $\frac{1}{2}$.
\label{fig:equi-tri-thin}}
\end{figure}
This tiling {\em is} locally finite.
Nevertheless, this example is not really satisfying. More precisely, this
tiling is not normal, since in this solution the perimeters of the 
triangles become arbitrary large. (Hence the inradii become
arbitrary small). So it seems natural to ask:
\begin{quest} \label{q:tritilnormal}
``Is there a normal tiling of the plane by pairwise noncongruent 
triangles of unit area?''
\end{quest}
This question was already asked by Nandakumar, in the form whether 
there is a tiling of the plane by pairwise noncongruent triangles 
all having unit area such that the perimeter of the triangles is 
bounded by some common constant. Theorem \ref{thm:tritil} below 
answers this question affirmatively.

One possible approach to find a solution is the following. If one 
can partition a set $S \subset \R^2$ into triangles of unit area 
such that (1) $S$ tiles the plane, and (2) the triangles can be 
distorted continuously, in a way such that all triangles in 
$S$ are distinct (but still having unit area), this solves the 
problem. We will illustrate this concept (where "triangle" is 
replaced by "quadrangle") in the proof of the following result.
\begin{figure}
\includegraphics[width=.20\textwidth]{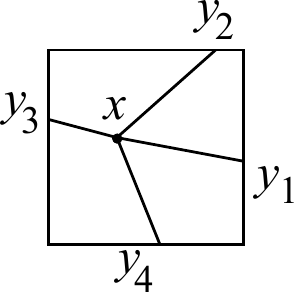}
\caption{Partitioning a square into 4 distinct quadrangles of
equal area. There are uncountably many possibilities for 
such a dissection. \label{fig:equi-quad-nonvtv}}
\end{figure}
\begin{thm} \label{thm:quad-til-nonvtv}
There is a normal tiling of the plane by pairwise non-congruent
convex quadrangles of unit area.
\end{thm}
\begin{proof}
Consider a square $Q$ of edge length 2. Let $x \in Q$ be a point
such that (1) for the distance $d$ of $x$ to the centre of $Q$ 
holds $0 < d < \frac{1}{10}$, and (2) $x$ is neither contained 
in the diagonals of $Q$ nor in the line segments connecting
mid-points of opposite edges of $Q$. Let $y_1$ be a point on 
the boundary of $Q$ such that for the distance $d'$ of $y$ to 
the mid-point of the edge containing $y_1$ holds $0<d'<\frac{1}{10}$.
(Here the requirements on $x$ and $y_1$ being close to the centre of
the squre resp.\ the centre of the edge are needed to exclude the 
possibility of non-convex quadrangles resp.\ triangles occurring.)

The choice of $x$ and $y_1$ determines three further unique points $y_2,
y_3, y_4$ on the boundary of $Q$ such that the line segments $xy_i$
($1 \le i \le 4$) partition $Q$ into four quadrangles of unit area. By the 
choice of $x$, avoiding the mirror axes of $Q$, it can always be achieved
that all quadrangles in a single partition of $Q$ are distinct.  
(In fact the author believes that no two congruent quadrangles
can occur in a partition where $x$ is not contained in the 
mirror axes of $Q$, but this might be tedious to prove. Here we
prefer rather to use the free parameters to achieve that all
quadrangles are distinct.) Figure \ref{fig:equi-quad-nonvtv} indicates 
such a partition.

The two coordinates determining $x$ can be changed continuously 
within a small range independently, yielding two free parameters.
One coordinate of $y_1$ can be changed continuously, too, within
some small range. Hence we obtain the desired tiling as follows:
Tile the plane $\R^2$ with copies of the square $Q$. Dissect each 
copy of $Q$ into four quadrangles of area 1, in some order. In each 
dissection, choose $x$ and $y_1$ such that the resulting quadrangles
have not shown up earlier in the construction. This is always 
possible since, in each step, there are only finitely many quadrangles 
constructed already, whereas there are uncountably many choices for 
$x$ and $y_1$.
\end{proof}
  
At this point it becomes obvious that Questions \ref{q:frage} and 
\ref{q:tritilnormal} lead to several variants. The example in 
the proof above yields a normal tiling, but in general not one 
that is vertex-to-vertex. One may ask the questions for triangles,
for quadrangles, for pentagons, and in each case with or without 
requiring ``equal perimeter'', or ``normal'', or ``vertex-to-vertex''.
The next section aims to give a systematic overview of the questions.

\section{Variants of the problem} \label{sec:more}

The general property we will require throughout the paper is that a tiling 
consists of convex tiles of unit area such that all tiles are pairwise distinct.
The tiles can be triangles (as in the original question), but also
quadrangles, rectangles, pentagons or hexagons. We may or may not
require additionally that all tiles have {\em equal perimeter}, or that the 
{\em perimeter is bounded} by some common constant, or that the tilings 
are {\em normal}, or just {\em locally finite}. Furthermore, it may be 
possible to construct a tiling analogous to the proof of Theorem 
\ref{thm:quad-til-nonvtv}, that is, by {\em tiling a tile} $S$ in infinitely 
many ways, where $S$ in turn can tile the plane. The connections 
between these properties is shown in the following diagram.
\begin{equation} \label{eq:implic} 
\left. \begin{array}{lc} 
\mbox{equal perimeter} & \Rightarrow \\
\mbox{tiling a tile}  & \Rightarrow \\
\end{array} 
\right\} \mbox{ perimeter is bounded }
\Leftrightarrow \mbox{ normal } \Rightarrow \mbox{ locally finite} 
\end{equation}
For instance, Equation \eqref{eq:implic} tells that if there is 
tiling obtained by tiling a tile $S$ in infinitely many ways, then 
the perimeters of the tiles in this tiling are bounded by some common 
constant. In turn, the latter is equivalent to the tiling being normal
(since all tiles are convex and have unit area), which in turn implies
(by \cite[3.2.1]{GS}) that the tiling is locally finite. 

These implications help to give an overview of the several variants of 
the questions. The following tables list, for each of the cases of triangles, 
convex quadrangles, convex pentagons, and convex hexagons, whether 
there is some tiling known fulfilling the properties in Equation 
\eqref{eq:implic}, and whether there is such a tiling that is even 
vertex-to-vertex. Because of the implications in Equation 
\eqref{eq:implic}, if there is "yes" in some column, then the entries 
above in the same column contain also a "yes". 

Note that ``not vtv'' is usually a weaker condition than
``vtv'', but a tiling by convex hexagons that is {\em not} vtv 
is much harder to find than one that is vtv.

\vspace{2mm}
 
\begin{tabular}{l|cc}
\hline
Triangles & vtv &  not vtv \\
\hline
locally finite & ? & yes\\
bounded perimeter & ? & yes\\
tiling a tile & ? & ? \\
equal perimeter & ? & ?\\
\hline
\end{tabular}
\quad
\begin{tabular}{l|cc}
\hline
Quadrangles & vtv &  not vtv \\
\hline
locally finite & yes & yes\\
bounded perimeter & yes & yes\\
tiling a tile & yes & yes\\
equal perimeter & ? & ?\\
\hline
\end{tabular}

\begin{tabular}{l|cc}
Pentagons & vtv &  not vtv \\
\hline
locally finite & ? & yes\\
bounded perimeter & ? & yes\\
tiling a tile & ? & yes\\
equal perimeter & ? & ?\\
\hline
\end{tabular}
\quad
\begin{tabular}{l|cc}
Hexagons & vtv & not vtv\\
\hline
locally finite & yes & ? \\
bounded perimeter & yes & ? \\
tiling a tile & yes & ? \\
equal perimeter & ? & ?\\
\hline
\end{tabular}

\vspace{2mm}

Theorem \ref{thm:quad-til-nonvtv} above proves the case ``quadrangles: tiling
a tile, not vtv'' (thus also the two cases above it in the same column in
the corresponding table). Theorem \ref{thm:tritil} below proves ``triangles: 
bounded perimeter, not vtv'', Theorem \ref{thm:quadtil} proves 
``quadrangles: tiling a tile, vtv'',  Theorem \ref{thm:pentil} proves ``hexagons: 
tiling a tile, vtv'', and Theorem \ref{thm:hextil} proves ``hexagons: tiling a tile, vtv''.

\section{Main results} \label{sec:tritil}

\begin{thm} \label{thm:tritil}
There is a normal tiling of the plane by pairwise non-congruent triangles 
of unit area.
\end{thm}
\begin{proof}
The idea of the proof is a refinement of the construction in
Figure \ref{fig:equi-tri-thin}. Basically we add additional fault
lines in each quadrant. Moreover, we make use of some free parameter
in some range, allowing for uncountably many choices, where in
each step of the construction only finitely many triangle shapes must 
be avoided.

Choose some constant $c$ big enough. This serves as the upper bound 
on the perimeter of the triangles. For our purposes $c=100$ will do.
Consider the upper right quadrant $Q_1$. Pick a point $x_0$ on 
the positive horizontal axis with $|x_0|<\frac{c}{3}$. Let $T_1$ be the 
unique triangle in $Q_1$ with vertices $0, x_0$ and area 1. (For this and
what follows compare Figure \ref{fig:equi-tri}.) Denote the third vertex
of $T_1$ by $y_1$. Choose $y_2$ on the horizontal axis such that the
triangle $T_2$ with vertices $x_0,y_1$ and $y_2$ has area 1. Continue 
zigzagging in this way between horizontal and vertical axis. I.e.,
choose $y_{i+1}$ on the axis not containing $y_i$ such that the triangle
$T_{i+1}$ with vertices $y_{i-1},y_i,y_{i+1}$ has area 1. Repeat this until the 
next triangle $T_{i+2}$ would have perimeter larger than $c$. Omit 
$T_{i+2}$. Pick $x_1$ such that the triangle $y_i, y_{i+1},x_1$ has area 1.

There are uncountably many choices for $x_1$. For the sake of symmetry 
let $x_1$ be close to the bisector $\{(x,x) \, | \, x \in \R \}$ of $Q_1$. 
Choose a half-line $\ell_1$ proceeding from $x_1$. There are 
uncountably many choices for $\ell_1$. Again, for the sake of
symmetry, let $\ell_1$ be close to the bisector of $Q_1$. 
\begin{figure}
\includegraphics[width=.65\textwidth]{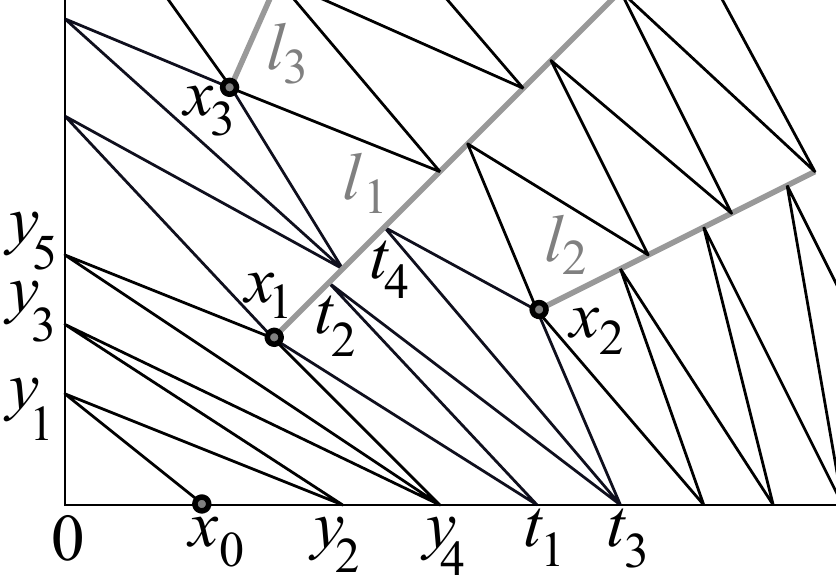}
\caption{Tiling the upper right quadrant $Q_1$ by pairwise 
non-congruent triangles of unit area and bounded perimeter. 
\label{fig:equi-tri}}
\end{figure}
Continue by zigzagging in two regions, between the horizontal axis
and $\ell_1$, and between the vertical axis and $\ell_1$. I.e., if
$y_{2k}$ is the last point on the horizontal axis, pick $t_1$ on
the horizontal axis such that the triangle $y_{2k},x_1,t_1$ has
area 1. Continue by choosing $t_2$ on $\ell_1$ such that the 
triangle $x_1,t_1,t_2$ has area 1 and so on, until the perimeter of 
the next triangle $t_i, t_{i+1},t_{i+2}$ would be larger than $c$. 
Omit this triangle. Choose $x_2 $ such that the new triangle 
$t_i,t_{i+1},x_2$ has area 1. Choose a half-line $\ell_2$ 
proceeding from $x_2$. Again there are uncountably many
choices for $x_2$ and $\ell_2$. 

Do the analogous construction in the upper region between $\ell_1$ and the 
vertical axis. Continue in this manner. Whenever a triangle occurs 
with perimeter larger than $c$ choose a new point $x_k$ and a new line 
$\ell_k$ dividing the old region into two.

The uncountability of choices for $x_k$ and $\ell_k$ ensures that
we can always avoid to add a triangle that is congruent to some 
triangle added earlier. Indeed, whenever we
are in the situation to choose $x_k$ and $\ell_k$ there are at most
countably many triangles constructed already. Hence $x_k$ can be chosen 
such that no triangle with vertex $x_k$ is congruent to an already
constructed one, and $\ell_k$ can be chosen such that no triangle
occurring in the two new regions defined by $\ell_k$ is congruent 
to an already constructed one. Hence the quadrant $Q_1$ can be tiled
by pairwise non-congruent triangles with area 1 and perimeter less 
than $c$.

The other quadrants can be tiled accordingly. Whenever a choice of
a new point and a new half-line happens there are uncountably many
possibilities, hence all (at most countably many) already constructed 
triangles can be avoided.
\end{proof}
\begin{thm} \label{thm:quadtil}
There is a normal vtv tiling of the plane by pairwise non-congruent 
quadrangles of unit area. The tiling consists of squares that are 
dissected into four distinct quadrangles of equal area.
\end{thm}
\begin{figure}
\includegraphics[width=.55\textwidth]{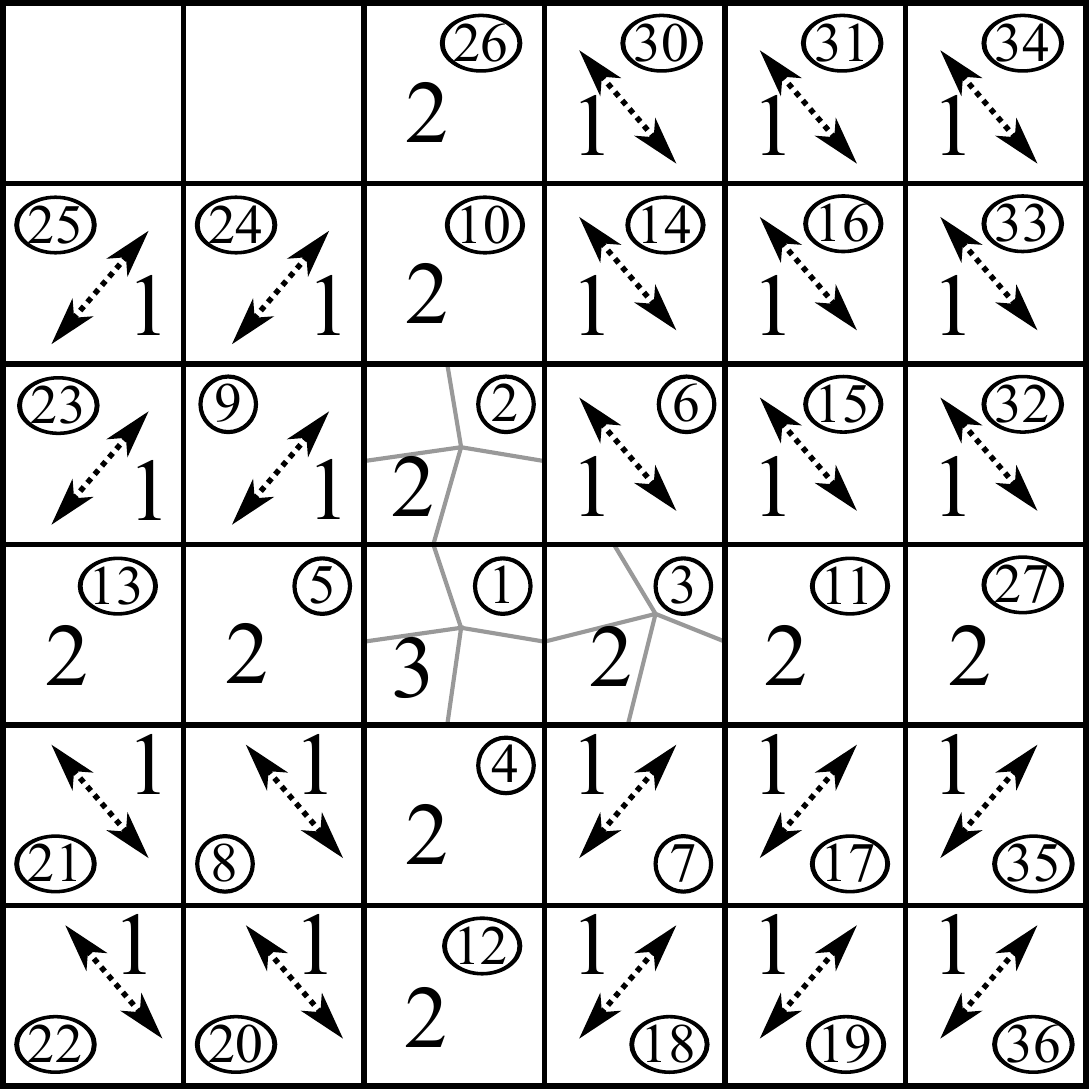}
\caption{Tiling a square with distinct quadrangles of unit area
(compare Figure \ref{fig:equi-quad-nonvtv}) can be done in 
a way such that partitions of adjacent squares are vertex-to-vertex.
Circled numbers indicate the order in which squares are added
consecutively, ordinary numbers indicate the degrees of freedom
in each square. A 2 means that the centre can be wiggled within
a small ball. A 1 means that the centre can be shifted along some
line by a small amount. The (approximate) directions of these lines 
are indicated by dashed line segments.  
\label{fig:equi-quad-vtv}}
\end{figure}
\begin{proof}
The idea is to use the construction in the proof of Theorem 
\ref{thm:quad-til-nonvtv}, adding (dissected) squares 
consecutively, using the degrees of freedom 
to achieve vertex-to-vertex in neighbouring squares. 
Figure \ref{fig:equi-quad-vtv} indicates the order in which
squares are added, and the degrees of freedom in the
dissection of each square.

Start with some square $S$, dissected as in the proof of
Theorem \ref{thm:quad-til-nonvtv}. There are three
degrees of freedom how to dissect $S$ into four quadrangles, 
two for placing the centre of dissection, one for a point on
the boundary. This square is indicated by a circled 1 in
Figure \ref{fig:equi-quad-vtv}. Add four more dissected squares 
adjacent to $S$, such that the quadrangles are vertex-to-vertex.
These squares are numbers 2 to 5 in the figure. 
In each of these squares there are still two degrees of freedom
for placing the centre. The third parameter is determined uniquely
by the vertex-to-vertex condition. Still one may use the two degrees
of freedom to avoid adding a quadrangles that is congruent to one
added already.

Now add four more squares (6 to 9), each one adjacent to two edges 
of squares 2 to 5, respectively. Now the position of two points of
the dissection are determined for each of the squares 6 to 9. 
Hence, by the area condition, the centre of the square is restricted to 
some line. Anyway, it can be shifted along a small segment of this line 
continuously. Hence there is still one free parameter that we can use 
to avoid adding a quadrangle that is congruent to some quadrangle 
added earlier.

In this way we continue filling the plane: add four squares along
the horizontal and vertical axes (the next step would be adding
squares 10 to 13
in the figure), add more squares to the pattern to complete a
square pattern. Proceeding in this way ensures that in each step 
there is at least one free parameter that can be used to avoid
adding a square congruent to one added earlier.
\end{proof}

\begin{thm} \label{thm:pentil}
There is a normal tiling of the plane by pairwise non-congruent 
pentagons of unit area. The tiling consists of hexagons that are 
dissected into three distinct pentagons of equal area.
\end{thm}
\begin{proof}
A regular hexagon of area three can be divided into three hexagons 
of unit area in uncountably many ways, compare the left part of
Figure \ref{fig:equi-pent-hex}.
\end{proof}
\begin{figure}
\includegraphics[width=.55\textwidth]{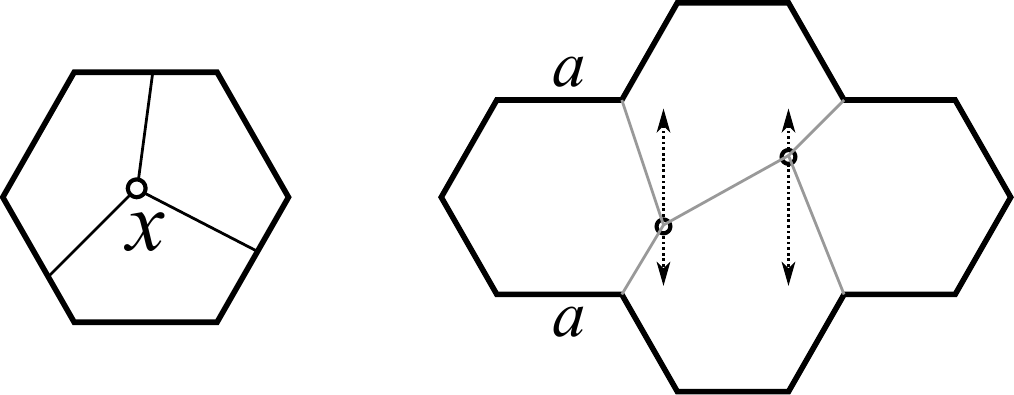}
\caption{A regular hexagon can be divided into three distinct pentagons 
of equal area in uncountably many ways (left). A non-convex 14-gon
can be divided into four distinct hexagons of equal area in uncountably
many ways. 
\label{fig:equi-pent-hex}}
\end{figure}
\begin{thm} \label{thm:hextil}
There is a normal vtv tiling of the plane by pairwise non-congruent 
hexagons of unit area. The tiling consists of non-convex 14-gons
that are dissected into four distinct hexagons of equal area.
\end{thm}
\begin{proof}
Consider a non-convex 14-gon assembled from three regular 
hexagons and a fourth hexagon that is obtained from a regular 
hexagon by stretching it slightly in the direction of one of
the edges, see Figure \ref{fig:equi-pent-hex} right. The longer
edges are labelled with $a$ in the figure. This 14-gon can be 
dissected into four hexagons of equal area. There is still
one parameter of freedom: one vertex of the dissection can
be shifted continuously along a line segment, the other
interior vertex of the dissection is then determined uniquely
by the area condition. 

The 14-gons yield a tiling of the plane: gluing 14-gons together
at their edges of length $a$ yields biinfinite strips. These strips
in turn can be assembled into a tiling.
\end{proof}

During working on the problem the author tried several 
approaches. Based on this experience we want to highlight
the following problems for further study.
\begin{enumerate}
\item Is there a compact convex region in the plane that 
can be tiled by non-congruent triangles of unit area in 
infinitely many (uncountably many) ways? 
\item Is there a compact region in the plane that (a)
can be tiled by non-congruent triangles of unit area in 
infinitely many (uncountably many) ways, and (b) tiles the plane? 
\item Is there a vertex-to-vertex tiling of the plane by pairwise 
non-congruent triangles of unit area?
\item Is there a vertex-to-vertex tiling of the plane by pairwise 
non-congruent triangles of unit area such that the perimeter of the 
triangles is bounded by some common constant?
\item Is there a tiling of the plane by pairwise 
non-congruent rectangles of unit area such that the perimeter of the 
rectangles is bounded by some common constant?
\end{enumerate}

\section*{Acknowledgments}
The author expresses his gratitude to R.~Nandakumar for providing 
several interesting problems. Special thanks to Jens Schubert for helpful 
discussions on this topic during a pleasant summer weekend in Bochum.

\end{document}